\documentclass{amsart}
\pdfoutput=1
\usepackage{hyperref}
\usepackage[alphabetic]{amsrefs}

\theoremstyle{plain}
\newtheorem{thm}{Theorem}
\newtheorem{defn}[thm]{Definition}
\newtheorem{rem}[thm]{Remark}
\newtheorem{prop}[thm]{Proposition}

\title{Complex structures on Three-point space}
\author{Suvrajit Bhattacharjee}
\address{Stat-Math Unit\\ Indian Statistical Institute\\ 203, B.T. Road\\ Kolkata-700108}
\email{suvra.bh\_r@isical.ac.in}
\author{Debashish Goswami}
\address{Stat-Math Unit\\ Indian Statistical Institute\\ 203, B.T. Road\\ Kolkata-700108}
\email{goswamid@isical.ac.in}
\thanks{D.G. is partially supported by J.C. Bose National Fellowship and Research Grant awarded by D.S.T. (Govt. of India)}
\dedicatory{\textbf{Dedicated to (late) Prof. Hida, Prof. Accardi and Prof. Volovich}}

\begin{document}

\begin{abstract}
We discuss notions of almost complex, complex and K\"{a}hler structures in the realm of non-commutative geometry and investigate them for a class of finite dimensional spectral triples on the three-point space. We classify all the almost complex structures on this non-commutative manifold, which also turn out to be complex structures, but none of them are K\"{a}hler in our sense.
\end{abstract}

\maketitle

\section{Introduction}
Almost complex, complex and K\"{a}hler structures play an important role in  differential and Riemannian geometry. As a generalization of classical Riemannian geometry, non-commutative geometry was formulated by Alain Connes \cites{MR1303779,MR1482228}. In this context, it is important and natural to formulate analogues of almost complex, complex and K\"{a}hler structures. However, it is rather surprising that except for a few attempts \cite{MR1695097} in the early days of the subject and more recent works by a few authors \cites{re,MR3720811,MR3720811,MR3073899,MR2773332,MR2838520} not too much has been done in this direction. In this article, our humble aim is to propose some formulation of these notions and investigate them in a very simple, finite dimensional non-commutative manifold on the three-point space. We classify all the almost complex structures on this non-commutative manifold, which also turn out to be complex structures, but none of them are K\"{a}hler in our sense.

\section{Complex Structures}

We consider a space made of three points $Y=\{1,2,3\}$. The algebra $\mathcal{C}(Y)$ of continuous functions is the direct sum $\mathcal{C}(Y)=\mathbb{C} \oplus \mathbb{C} \oplus \mathbb{C}$ and any element $f \in \mathcal{C}(Y)$ is a triplet of complex numbers $(f_1,f_2,f_3)$, with $f_i=f(i)$ the value of $f$ at the point $i$. The functions $\chi_i$ defined by $\chi_i(j)=\delta_{ij}, \quad i,j=1,2,3$ form a $\mathbb{C}$-basis for $\mathcal{C}(Y)$.\\

The following is a simple conceptual tool that will be needed later on.

\begin{defn} \cite{MR3073899}
Let $A$ be a $\ast$-algebra and $E$ be an $A$ bi-module. The conjugate bi-module $\overline{E}$ is defined by declaring:
\begin{itemize}
\item[i)] $\overline{E}=E$ as abelian groups;
\item[ii)] We write $\overline{e}$ for an element $e \in E$ when we consider it as an element of $\overline{E}$;
\item[iii)] The bi-module operations for $\overline{E}$ are $a \cdot \overline{e}=\overline{e \cdot a^*}$ and $\overline{e} \cdot a=\overline{a^* \cdot e}.$
\end{itemize}
\end{defn}

\begin{rem}
If $\theta : E \rightarrow F$ is any morphism then we define $\overline{\theta} : \overline{E} \rightarrow \overline{F}$ by $\overline{\theta}(\overline{e})=\overline{\theta(e)}$.
\end{rem}

We now introduce the non-commutative analogue of K\"{a}hler differentials for commutative algebras.

\begin{defn}\cite{MR3073899}
Let $A$ be an arbitrary $\mathbb{C}$-algebra. Let $\Omega_{u}^1A$ denote the kernel of the multiplication map $\mu : A \otimes A \rightarrow A$. We then define the universal differential graded algebra over $A$ to be the tensor algebra \[\Omega_uA=T_A(\Omega_u^1A)\] endowed with the unique degree one derivation such that \[d(a):= 1 \otimes a - a \otimes 1\] for $a \in A$.
\end{defn}

\begin{defn}\cite{MR3073899}
A differential calculus or differential structure on $A$ is a differential graded algebra $(\Omega A, d)$ that is a quotient of $(\Omega_uA,d)$ by a differential graded ideal whose degree-zero component is zero.
\end{defn}

It follows that $\Omega A$ is generated by $A$ and $\Omega^1A$, and $\Omega^nA$ is the $\mathbb{C}$-span of \[\{\ a_0da_1\wedge \cdots \wedge da_n \mid a_0\dots a_n \in A\}.\]

The notion of a differential $\ast$-calculus first appeared in \cite{MR994499}.

\begin{defn}\cite{MR3073899}
A differential calculus $(\Omega A, d)$ on a $\ast$-algebra $A$ is compatible with the star operation on $A$ if the star operation on $\Omega^0A=A$ extends to an involution $\xi \mapsto \xi^*$ on $\Omega A$ that preserves the grading and has the property that $d(\xi)^*=d(\xi^*)$ and $(\xi \wedge \eta)^*=(-1)^{\lvert \xi \rvert \lvert\eta \rvert}\eta ^* \wedge \xi^*$ for all homogeneous $\eta, \xi \in \Omega A.$ When the conditions hold we call $(\Omega A, d, \ast)$ a differential $\ast$-calculus on $A$.
\end{defn}

\begin{rem}
The definition implies $(a\xi b)^*=b^* \xi ^* a^*$ for all $a,b \in A$ and $\xi \in \Omega A.$	
\end{rem}

\begin{prop}\cite{MR3073899}
The map $\star : \Omega A \rightarrow \overline{\Omega A}$ defined by $\star(\xi)=\overline{\xi^*}$ is an $A$-bimodule homomorphism.
\end{prop}

We will now produce a differential $\ast$-calculus on the $\ast$-algebra $\mathcal{C}(Y)$ of continuous functions on the three-point space.\\

The notion of a spectral triple is the non-commutative analogue of a manifold. It is a basic ingredient in non-commutative geometry as developed in \cite{MR823176}.

\begin{defn}\cites{MR1303779, MR1482228} \label{8}
A spectral triple $(A,H,D)$ is given by a $\ast$-algebra with a faithful representation $\pi : A \rightarrow B(H)$ on the Hilbert space $H$ together with a self-adjoint operator $D=D^*$ on $H$ with the following properties:

\begin{itemize}
\item[i)] The resolvent $(D-\lambda)^{-1}$, $\lambda \not \in \mathbb{R}$, is a compact operator on $H$;
\item[ii)] $[D,a]:= D\pi(a)-\pi(a)D \in B(H)$, for any $a \in A$.
\end{itemize}

\end{defn}

Given a spectral triple $(A,H,D)$ one constructs \cite{MR1482228} a compatible differential calculus on $A$ by means of a suitable representation of the universal algebra $\Omega_uA$ in the algebra of bounded operators on $H$. The map\[\pi_u : \Omega_uA \rightarrow B(H)\] given by \[\pi_u(a_0da_1\cdots da_p):= a_0[D,a_1]\cdots [D,a_p], \qquad a_j \in A\] is a $\ast$-homomorphism of algebras.

\begin{prop}\cite{MR1482228}
Let $J_0:= \oplus_p J_0^p$ be the graded two sided ideal of $\Omega_uA$ given by \[J^p_0:=\{ \omega \in \Omega^p_uA \mid \pi_u(\omega)=0\}.\] Then $J:=J_0+dJ_0$ is graded differential ideal of $\Omega_uA$.
\end{prop}

\begin{rem}
The elements of $J$ are called junk forms. 
\end{rem}

\begin{defn}
The differential graded algebra of Connes' forms over the algebra $A$ is defined by \[\Omega_DA:=\Omega_uA/J \cong \pi_u(\Omega_uA)/\pi_u(J).\]
\end{defn}

\begin{rem}
This is a differential $\ast$-calculus on the $\ast$-algebra $A$.
\end{rem}

Let us now describe a spectral triple on the three-point space $Y=\{1,2,3\}$. This is a special case of a class of spectral triples considered in \cites{MR2327719,MR1727499} on compact metric spaces.

\begin{prop}\label{13} 
Put $A=\mathcal{C}(Y)$ and $H=\mathbb{C}^2_{12} \oplus \mathbb{C}^2_{23} \oplus \mathbb{C}^2_{13}$ (the subscript $ij$ says that the Hilbert space is along the ``edge" connecting the point $i$ with $j$). Define $\pi : A \rightarrow B(H)$ by \[\pi(f)=
\begin{bmatrix}
f(1) & 0\\
0 & f(2)\\
\end{bmatrix}_{12} \oplus 
\begin{bmatrix}
f(2) & 0\\
0 & f(3)\\
\end{bmatrix}_{23} \oplus
\begin{bmatrix}
f(1) & 0\\
0 & f(3)\\
\end{bmatrix}_{13},\] for $f \in A$. And finally, define the operator $D$ as \[D=
\begin{bmatrix}
0 & 1\\
1 & 0\\
\end{bmatrix}_{12} \oplus
\begin{bmatrix}
0 & 1\\
1 & 0\\
\end{bmatrix}_{23} \oplus
\begin{bmatrix}
0 & 1\\
1 & 0\\
\end{bmatrix}_{13}.\] Then $(A,H,D)$, as constructed above, is a spectral triple on the three-point space.
\end{prop}

\begin{proof}
The conditions of \hyperref[8]{Definition \ref*{8}} are satisfied since $H$ is a finite dimensional Hilbert space ($D$ is manifestly self-adjoint).
\end{proof}

\begin{rem}
Let us denote the differential graded algebra of Connes' forms over the algebra $A=\mathcal{C}(Y)$ simply by $\Omega.$
\end{rem}

\begin{thm}\label{15}
Let $(A,H,D)$ be the spectral triple on the three-point space described in \hyperref[13]{Proposition \ref*{13}}. Then  $e_i=[D,\chi_i], i=1,2$ is a free right basis for $\Omega^1$. The bi-module structures are given by 
\begin{itemize}
\item $\chi_i e_i=e_i(1-\chi_i)$
\item $\chi_i e_j=-e_i \chi_j, \qquad i\neq j$
\end{itemize}
and 
\begin{itemize}
\item $e_i \chi_i=(1-\chi_i)e_i$
\item $e_i \chi_j=-\chi_i e_j \qquad i\neq j$
\end{itemize}
\end{thm}

\begin{proof}
We note that, by definition, the space of 1-forms consists of bounded operators on $H$ of the form $\sum_j a_0^j[D,a_1^j]$, where $a_i^j \in A$. We recall that $\chi_i, \quad i=1,2,3$ form a $\mathbb{C}$ basis of $A$ and satisfies $\chi_1+\chi_2+\chi_3=1$. Since $[D,1]=0$, we get the first conclusion. The bi-module structure follows from Leibniz rule and the observations $\chi_i^2=\chi_i$, $\chi_i\chi_j=0$. 
\end{proof}

\begin{rem}
The basis described above is also a left basis for $\Omega^1$.
\end{rem}

\begin{rem}
It can be shown by computation that there are no junk forms. Also the higher spaces of forms are finite dimensional vector spaces.
\end{rem}

The following definition is the beginning of non-commutative complex geometry and one can go very far mimicking the classical theory along the lines \cite{MR2093043}.

\begin{defn}\cites{MR3073899, MR3428362}
Let $(\Omega A, d, \ast)$ be a $\ast$-differential calculus on $A$. An almost complex structure on $(\Omega A, d, \ast)$ is a degree zero derivation $J : \Omega A \rightarrow \Omega A$ such that 
\begin{itemize}
\item[i)] $J$ is identically $0$ on $A$ and hence an $A$-bimodule endomorphism of $\Omega A;$
\item[ii)] $J^2=-1$ on $\Omega^1A;$ and
\item[iii)] $J(\xi^*)=J(\xi)^*$ for $\xi \in \Omega^1A$ ($J$ preserves $\ast$, i.e., $\overline{J} \star=\star J$).
\end{itemize}
\end{defn}

Because $J^2=-1$ on $\Omega^1A$, there is an $A$-bimodule decomposition \[\Omega^1A=\Omega^{1,0}A \oplus \Omega^{0,1}A\] where \[\Omega^{1,0}A=\{\ \omega \in \Omega^1A \mid J\omega=\iota \omega \}\] and \[\Omega^{0,1}A=\{\ \omega \in \Omega^1A \mid J\omega=-\iota \omega\}.\] 

Condition iii) implies $(\Omega^{0,1}A)^*=\Omega^{1,0}A.$ \footnote{The map $\star : \Omega^{1,0}A \rightarrow \overline{\Omega^{0,1}A}$ is an isomorphism of $A$ bi-modules.}\\

For all $p,q \geq0$ we define \[\Omega^{p,q}A:=\{\ \xi \in \Omega^{p+q}A \mid J\xi=(p-q)\iota \xi\}.\] Elements in $\Omega^{p,q}A$ are called $(p,q)$ forms. It is a theorem that \[\Omega^nA=\bigoplus_{p+q=n} \Omega^{p,q}A.\] For all pairs of non-negative integers $(p,q)$, let \[\pi^{p,q} : \Omega^{p+q}A \rightarrow \Omega^{p,q}A\] be the projections associated to the direct sum decomposition $\Omega^nA=\bigoplus_{p+q=n} \Omega^{p,q}A$.

\begin{rem}\label{19}
We call an endomorphism $J : \Omega^1A \rightarrow \Omega^1A$ satisfying the above conditions a first order almost complex structure. In most of the examples studied so far, one defines the endomorphism on $\Omega^1A$ and then extends it to whole of $\Omega A$ using the derivation property and some ``basis" of higher forms. In our case we use the free basis for $\Omega^1$ and vector space basis for higher forms.
\end{rem}

\begin{defn}\cites{MR3073899, MR3428362}
Let \[\partial : \Omega^{p,q}A \rightarrow \Omega^{p+1,q}A\] be the composition $\pi^{p+1,q}d : \Omega^{p,q}A \rightarrow \Omega^{p+1,q}A$. Let \[\bar{\partial} : \Omega^{p,q}A \rightarrow \Omega ^{p,q+1}A\] be the composition $\pi^{p,q+1}d : \Omega^{p,q}A \rightarrow \Omega^{p,q+1}A.$
\end{defn}

The following is the analogue of Newlander-Nirenberg theorem \cite{MR2093043}.

\begin{defn}\cites{MR3073899, MR3428362}
An almost complex structure $J$ on $(\Omega A, d, \ast)$ is integrable if $d\Omega^{1,0}A \subset \Omega^{2,0}A \oplus \Omega^{1,1}A$.
\end{defn}

\begin{defn}\cites{MR3073899, MR3428362}
A complex structure on $(\Omega A, d, \ast)$ is an almost complex structure $J$ which is integrable.
\end{defn}

Now we explicitly determine all the complex structures on the three-point space.

\begin{thm}\label{23}
Let $(A,H,D)$ be the spectral triple on the three-point space as described in \hyperref[13]{Proposition \ref*{13}}. Then there are 8 complex structures for the calculus obtained in \hyperref[15]{Theorem \ref*{15}}, as enumerated below:
\begin{itemize}
\item 
$i\begin{bmatrix}
1-2\chi_3 & 0\\
2\chi_1 & 2\chi_2-1
\end{bmatrix}, \qquad
-i\begin{bmatrix}
1-2\chi_3 & 0\\
2\chi_1 & 2\chi_2-1
\end{bmatrix}$;\\
\item 
$i\begin{bmatrix}
1-2\chi_3 & -2\chi_2\\
2\chi_1 & 2\chi_3-1
\end{bmatrix}, \qquad
-i\begin{bmatrix}
1-2\chi_3 & -2\chi_2\\
2\chi_1 & 2\chi_3-1
\end{bmatrix}$;\\
\item 
$i\begin{bmatrix}
2\chi_1-1 & 2\chi_2\\
0 & 1-2\chi_3
\end{bmatrix}, \qquad
-i\begin{bmatrix}
2\chi_1-1 & 2\chi_2\\
0 & 1-2\chi_3
\end{bmatrix}$;\\
\item
$i\begin{bmatrix}
2\chi_1 & 0\\
0 & 1-2\chi_2
\end{bmatrix}, \qquad
-i\begin{bmatrix}
2\chi_1 & 0\\
0 & 1-2\chi_2
\end{bmatrix}.$
\end{itemize}
\end{thm}

\begin{proof}
By \hyperref[19]{Remark \ref*{19}}, we let $J$ be a first order almost complex structure. Let $Je_i=e_1J_{1i}+e_2J_{2i}, \quad i=1,2$ and extend right linearly. Then $J$ is a left module morphism reads as $J(\chi_ie_j)=\chi_iJ(e_j)$ (we use the bi-module rules to take the $\chi_i$ to the other side). In coordinates, $J^2=-1$ is $\sum_{j,k}e_kJ_{kj}J_{ji}=e_i$. Finally, $J$ preserves $\ast$ reads as $\sum_j e_jJ_{ji}=\sum_j\overline{J_{ji}}e_j$.\\

Now comparing coefficients and solving for $J_{ij}$ from the above equations, we get the first order almost complex structures. Surprisingly, there are no more first order almost complex structures than the listed ones. Next we extend these according to \hyperref[19]{Remark \ref*{19}} and note that the restriction of any almost complex structure to the space of one forms has to be one of these and since there is only one way to (derivation property) extend the first order ones, we get all the almost complex structures.\\

For the integrability condition, we check it individually case by case. For example, for $J=i
\begin{bmatrix}
2\chi_1 & 0\\
0 & 1-2\chi_2
\end{bmatrix}$, $\Omega^{1,0}$ consists of elements of the form $\alpha(e_1\chi_1)+\beta(e_2\chi_1)+\gamma(e_2\chi_3), \quad \alpha, \beta, \gamma \in \mathbb{C}$.\\ 

We apply $D$ to an element of that form, followed by $J$, and find that the result is $0$, i.e., the element lies in $\Omega^{1,1}$. So this $J$ is integrable. And similarly for the other $J$'s, which concludes the proof.
\end{proof}

\section{K\"{a}hler Structures}

Let us now go over to a possible non-commutative version of K\"{a}hler geometry. We begin with a basic ingredient in the classical theory, namely, that of a metric.

\begin{defn}
Let $(\Omega A,d,\ast)$ be a $\ast$-differential calculus on $A$. A metric $g$ on $A$ is a non-degenerate\footnote{Non-degeneracy means $g(\xi, \_)$ induces a right $A$ module isomorphism $\Omega^1A \rightarrow (\Omega^1A)'$, the $()'$ denoting the right $A$ dual.} bi-module morphism \[g : \Omega^1 A \otimes_A \Omega A^1 \rightarrow A\] and a hermitian metric is a non-degenerate bi-module morphism \[h : \Omega^1A \otimes_A \overline{\Omega^1A} \rightarrow A.\] 
\end{defn}

\begin{rem}
Given a hermitian metric $h : \Omega^1A \otimes_A \overline{\Omega^1A} \rightarrow A$, we get a metric $g : \Omega^1 A \otimes_A \Omega A^1 \rightarrow A$ by defining $g=h\circ(id \otimes \star)$. We say $h$ is induced by $g$.
\end{rem}

\begin{rem}
Assume that $\Omega^1A$ is projective as a left module. Then using a dual basis, we get a bi-module morphism \[\mathfrak{g} : A \rightarrow \Omega^1 A \otimes_A \Omega A^1.\] We identify the morphism with the element $\mathfrak{g}(1) \in \Omega^1 A \otimes_A \Omega A^1$. We call this element the inverse pf $g$
\end{rem}

Along the classical lines we have,

\begin{defn}
Let $(\Omega A,d,\ast)$ be a $\ast$-differential calculus with almost complex structure $J$. We say that a metric $g : \Omega^1 A \otimes_A \Omega A^1 \rightarrow A$ (respectively, a hermitian metric $h : \Omega^1A \otimes_A \overline{\Omega^1A} \rightarrow A$) is compatible with $J$ if $g \circ (J \otimes J)=g$ (respectively, $h \circ (J \otimes \overline{J})=h$).
\end{defn}

\begin{rem}
If $g$ is induced by $h$ and $h$ is compatible with $J$, then $g$ is automatically compatible with $J$.
\end{rem}

\begin{rem}
If a metric $g$ is compatible with $J$ then $g$ is identically $0$ on $\Omega^{1,0}A \otimes_A \Omega^{1,0}$ and $\Omega^{0,1}A \otimes_A \Omega^{0,1}A$.
\end{rem}

The very algebraic definition given in \cite{MR2093043} has the following counterpart.

\begin{defn}\label{30}
Let $(\Omega A,d,\ast)$ be a $\ast$-differential calculus with a complex structure $J$ and $g$ be a compatible metric on $A$. Assume $\Omega^1A$ is projective as a left module. Let $\mathfrak{g} \in \Omega^1 A \otimes_A \Omega A^1$ be the corresponding inverse of $g$. Then the fundamental form $\omega$ is defined to be the form $\wedge(J \otimes id)(\pi^{1,0}\otimes \pi^{0,1})(\mathfrak{g}) \in \Omega^{1,1}A$. The metric is said to be K\"{a}hler if $d\omega=0$.
\end{defn}

Unfortunately, for the three-point space we have the following

\begin{thm}
For the spectral triple $(A,H,D)$ on the three-point space as above in \hyperref[13]{Proposition \ref*{13}} and the calculus obtained in \hyperref[15]{Theorem \ref*{15}}, there are no compatible K\"{a}hler metric for the complex structures enumerated in \hyperref[23]{Theorem \ref*{23}}.
\end{thm}

\begin{proof}
This is also proved by case by case analysis. We outline the overall scheme.\\

We first find $J$ compatible metric $g$. Let $g(e_i,e_j)=g_{ij}$. Since $e_i, i=1,2$ is both a right and left basis for $\Omega^1$ and $g$ is a bi-module morphism, $g$ is determined by $g_{ij}$'s. The non-degeneracy condition turns into the invertibility of the matrix $\begin{bmatrix} g_{11} & g_{12} \\ g_{21} & g_{22}\end{bmatrix}$ (since $A$ is commutative, this is equivalent to the determinant $g_{11}g_{22}-g_{12}g_{21}$ being an unit).\\

Then the compatibility condition reads as $g_{ij}=g(\sum_ke_kJ_{ki},\sum_le_lJ_{lj})$. We take the $J_{ki}$'s to the right entry of $g$ and then use the bi-module rule to get those out of $g$. Solving, we get $J$ compatible metrics.\\

The inverse $\mathfrak{g}$ of $g$ takes the form $\sum_i e_i \otimes (\sum_je_jg^{ji})$, where $\begin{bmatrix}g^{11} & g^{12}\\ g^{21} & g^{22}\end{bmatrix}$ is the inverse of $\begin{bmatrix} g_{11} & g_{12} \\ g_{21} & g_{22}\end{bmatrix}$.\\

According to \hyperref[30]{Defintion \ref*{30}}, we apply the projections and then $J$ on the first tensorand (which is just multiplication by $i$) and multiply them, thus getting the fundamental form. The condition $d\omega=0$ then contradicts the non-degeneracy condition of $g$ as can be seen by direct computation.
\end{proof}

\begin{rem}
In the paper \cite{re} a similar definition of the fundamental form was proposed. But in \cite{MR3720811}, it is defined more conceptually and many of the standard classical theorems are proved in that paper. It is of interest to know whether the three point space admits a K\"{a}hler structure in the set-up of the paper \cite{MR3720811}.
\end{rem}

\begin{rem}
It remains to study several other examples. The quantum homogeneous spaces arising from quantum groups are studied in \cites{MR3428362, MR3720811} and are still being studied by several people. The Podles' sphere \cite{MR919322} and higher projective spaces are studied in \cites{MR2838520, MR2773332}. We wish to study the non-commutative complex geometry on the very basic example of a non-commutative manifold, namely, the non-commutative torus \cite{MR623572}. There has not been done much for this particular example except for the papers \cites{MR1977884, MR2054986, MR3282309} that we know of.
\end{rem}

\end{document}